\documentclass[amsart]{article}
\usepackage[left=2.5cm, right=2.5cm, top=3cm, bottom=3cm]{geometry} 
\usepackage[utf8]{inputenc}

\usepackage[english]{babel}
\usepackage{amssymb}
\usepackage{graphicx}
\usepackage[tbtags]{amsmath}
\usepackage{amsfonts}
\usepackage{mathrsfs}
\usepackage{xcolor}
\usepackage{commath}
\usepackage{amsthm}
\usepackage{lineno}

\newtheorem{theorem}{\quad Theorem}[section]

\newtheorem{definition}[theorem]{\quad Definition}

\newtheorem{lemma}[theorem]{\quad Lemma}

\newtheorem{proposition}[theorem]{\quad Proposition}

\newtheorem{remark}[theorem]{\quad Remark}

\newtheorem{thm}{Theorem}[section]

\newtheorem{deff}[thm]{Definition}

\title{On the first eigenvalue of a nonlinear Schr\"{o}dinger type equation}

\author{
Ardra A \footnote{The author was supported by the Department of Science and Technology INSPIRE Fellowship.}
}

\date{}
\begin{document}
	\maketitle
	
	\begin{abstract}
		We consider an eigenvalue problem for the generalized nonlinear Schr\"{o}dinger type operator with the Robin boundary condition as given below. 
 \begin{equation*}
 	\label{ab-Robin p-Laplace evp with potential term_intro}
 	\left\{
 	\begin{split}
 	-\Delta_p u+V(x)|u|^{p-2}u&=\lambda  |u|^{p-2}u\quad &&\mathrm{in} ~\Omega,\\
 	|\nabla u|^{p-2}\frac{\partial u}{\partial\eta}+\beta|u|^{p-2}u&=0\quad &&\mathrm{on}~\partial\Omega,
 	\end{split}
 	\right.
 	\end{equation*}
 	where $\Delta_p u := \operatorname{div}(|\nabla u|^{p-2}\nabla u)$ is the $p$-Laplace operator, $\Omega $ is a bounded domain in $\mathbb{R}^n$ with smooth boundary, $V \in C^1(\mathbb{R}^n),$ $ \eta $ denotes the outward unit normal, and $ \beta $ is a positive real constant. 
We study the properties of its first eigenvalue with respect to the potential $V$, the boundary parameter $\beta$ as well as the domain. 
    First, we establish some properties of the smallest eigenvalue $\lambda_1(V)$ with respect to the potential. We then prove the differentiability of $\lambda_1(V)$ with respect to the Robin boundary parameter $\beta$ and give an explicit formula for this derivative, which is then used to investigate some monotonicity properties of $\lambda_1(V).$  We also obtain a shape derivative formula for the smallest eigenvalue. Using these derivatives, we also study domain monotonicity properties of the first eigenvalue.  
	\end{abstract}
 \noindent
 {\bf Mathematics Subject Classification (2020):} {  35Q55, 47G40, 35J92, 35P30, 35P99, 49R05, 35P15,  47J10}.\\
 {\bf Keywords:}
 $p$-Laplacian, Schr\"{o}dinger type equation, eigenvalue problem, Robin boundary condition, shape derivative, Hadamard’s perturbations, domain monotonicity,  nonlinear eigenvalues,  elliptic equations.
 
	\section{Introduction}

We consider a nonlinear Schr\"{o}dinger type equation for the $p$-Laplacian as stated below:
 \begin{equation}
	\label{Robin p-Laplace evp with potential term}
	\left\{
	\begin{split}
	-\Delta_p u+V(x)|u|^{p-2}u&=\lambda  |u|^{p-2}u\quad \mathrm{in} ~\Omega,\\
	|\nabla u|^{p-2}\frac{\partial u}{\partial\eta}+\beta|u|^{p-2}u&=0\quad\mathrm{on}~\partial\Omega,
	\end{split}
	\right.
	\end{equation}
	where $\Omega $ is a bounded domain in $\mathbb{R}^n$ with smooth boundary, $V \in C^1(\mathbb{R}^n),$  $ \eta $ denotes the outward unit normal, and $ \beta $ is a positive real constant. 

\begin{definition} 
		Let $(u,\lambda)\in W^{1,p}(\Omega) \times \mathbb{R}$ such that $$ \int_{\Omega} |\nabla u |^{p-2}\nabla u \cdot \nabla \phi ~dx +\int_{\Omega} V(x)|u|^{p-2}u \phi ~dx + \int_{\partial \Omega} \beta |u|^{p-2}u\phi ~ds = \lambda\int_{\Omega} |u|^{p-2}u\phi ~dx$$ for all $\phi\in W^{1,p}(\Omega).$ Then $(u,\lambda)$ is a weak solution of \eqref{Robin p-Laplace evp with potential term}. 
\end{definition}

Fragnelli, Mugnai, and Papageorgiou (see Proposition 2.4, \cite{fragnelli2016brezis}) have shown that  there exists a smallest eigenvalue $\lambda_1$ for \eqref{Robin p-Laplace evp with potential term} when $V \in L^\infty (\Omega).$ It is characterized as  
\begin{equation*}
    \lambda_1(V)=\displaystyle{\inf_{ u \in W^{1,p}(\Omega) }} \left\{\int_{\Omega} |\nabla u |^{p} ~dx + \int_{\Omega} V(x)|u|^{p} ~dx + \int_{\partial \Omega} \beta |u|^{p} ~ds  : \left \| u \right \|_p=1 \right\}.
\end{equation*}
They also prove that it is simple, isolated, and the corresponding eigenfunction is in $C^{1,\alpha}(\overline\Omega)$ for some $\alpha \in (0,1)$ (see Proposition 2.4, \cite{fragnelli2016brezis}). It is also known that other eigenfunctions change sign (see Proposition 2.4, \cite{fragnelli2016brezis}).  

    
We are interested in investigating the behavior of its first eigenvalue. Eigenvalue problems for the Schr\"{o}dinger operator, which is an elliptic differential operator of the form $-\Delta + V(x),$ where $V(x)$ is a potential,  are among fundamental problems in mathematical physics (\cite{dunford1964linear}, \cite{dunford1988linear} and \cite{dunford1971linear}). \textcolor{black}{In quantum mechanics, the energy levels of a quantum particle in the potential energy $V$ are represented by the eigenvalues of the Schrödinger operator. In particular, $\lambda_1(V)$ is known as the ground state (see chapter 8 of \cite{henrot2006extremum}).} Minimization of eigenvalues arises in quantum scattering theory and the study of bulk matter. Extremum problems for the eigenvalues of such operators have been of much interest in the past years. See Chapter 8 in \cite{henrot2006extremum},  \cite{ashbaugh1987maximal} and \cite{bonder2006optimization}.

In this paper, first we study properties of the first eigenvalue of the nonlinear Schr\"{o}dinger problem \eqref{Robin p-Laplace evp with potential term} with respect to various parameters. We establish some properties of the first eigenvalue in terms of the
potential. We also study the differentiabilty of the first eigenvalue with respect to the Robin boundary
parameter. We also obtain a shape derivative formula for the first eigenvalue. We use this to deduce domain monotonicity properties for its first eigenvalue.

We show that the first eigenvalue satisfies the properties as given in the theorem below.
\begin{theorem}\label{properties Schrodinger p-Laplace eval}
$ \lambda_1(V)$ satisfies the following properties.
\begin{itemize}
     \item[(i)]  $ \lambda_1(V)$ is increasing with respect to $V.$
    \item[(ii)]  $ \lambda_1(V)$ is continuous with respect to $V$.
 \item[(iii)] If  $ \lambda_1(V) >0$, we can find a constant  $M>0$ such that
\begin{equation}
\label{property 3}
 M\|u\|^p_{W^{1,p}(\Omega)}\leq \int_{\Omega}|\nabla u|^p~dx+\int_{\Omega}V|u|^p~dx+\int_{\partial\Omega}\beta|u|^p~ds\quad \forall u\in W^{1,p}(\Omega).
\end{equation}
\end{itemize}
\end{theorem}

\textcolor{black}{We are also interested in studying the behavior of $\lambda_1(V)$ in terms of $\beta.$ In \cite{antunes2013asymptotic} and \cite{filinovskiy2014eigenvalues}, the authors have found the derivative of the first eigenvalue of the Robin Laplace eigenvalue problem with respect to the Robin boundary parameter. Their proofs rely upon the fact that the Laplace operator is linear and self-adjoint. For a fixed potential $V \in C^1(\mathbb{R}^n),$  we show that $\lambda_1(V)$ is differentiable with respect to the Robin boundary parameter $\beta$ and  obtain the formula for the derivative as   
   $$ \frac{d \lambda_1}{d\beta} = \frac{\int_{\partial \Omega} |u_{\beta}|^{p} ~ds }{ \int_{\Omega} |u_{\beta}|^{p} ~dx}, $$
where $u_{\beta}$ is the eigenfunction corresponding to $\lambda_1(V)$ normalized by $\|u_\beta\|_{\infty} = 1.$ A part of this proof uses some techniques from \cite{filinovskiy2014eigenvalues}, but since our operator is not linear or self-adjoint, we use a regularity result and Arzel\`{a} Ascoli theorem to prove it. Our proof is an alternate proof for the Robin Laplace eigenvalue problem as well. If $V=0,$ it also gives a formula for the derivative of the first Robin eigenvalue of the $p$-Laplacian with respect to $\beta.$ This also implies that the first eigenvalues are strictly increasing with respect to $\beta.$}

Understanding how energy functionals of differential equations involving the $p$-Laplacian depends on the domain is a fascinating area of study. Domain derivatives, also known as shape derivatives are used in such studies. Shape derivatives are obtained using the standard technique of Hadamard’s method of variations.  Shape derivative formula for the first Robin eigenvalue of the $p$-Laplacian was obtained in \cite{mallick2022shape} using this technique. \textcolor{black}{Using similar techniques, we obtain a shape derivative formula for the first eigenvalue of \eqref{Robin p-Laplace evp with potential term} as stated below. }
\begin{theorem} \label{Shape derivative formula-Schrodinger}
    Let $\Omega \subseteq \mathbb{R}^n$ be a bounded domain in the H$\ddot{\text{o}}$lder class $C^{2,\alpha},$  $ \Omega_t :=\{y=x+tv(x)+o(t):x\in \Omega, |t| \text{ sufficiently small}\}$ where $v$ is a $C^{2,\alpha}$ vector field,  and $H$ denote the mean curvature of $\partial \Omega .$ Then for the eigenvalue $\lambda_1(t):=\lambda_1(\Omega_t)$ of the Robin $p$-Laplace eigenvalue problem with potential $V \in C^1(\mathbb{R}^n),$ we have
   \begin{equation}
        \dot{\lambda}_1(0)  =\int_{\partial \Omega} \{\left[|\nabla u|^p + V|u|^p - \lambda_1 (0) |u|^p + \beta |u|^p (n-1)H + p \beta u \left( |u|^{p-2} \nabla u \cdot \eta \right) \right] (v.\eta)  \} ds.
   \end{equation}
\end{theorem}
\noindent The presence of the potential term makes the computation challenging.

The first Dirichlet eigenvalue $\lambda_1^D$ for the $p$-Laplacian exhibits domain monotonicity under domain inclusions, that is, $\lambda_1^D(\Omega_1)\leq \lambda_1^D(\Omega_2) $ for smooth bounded domains $\Omega_2 \subseteq \Omega_1$. Robin eigenvalues do not have this property in general. A  counterexample can be found in \cite{giorgi2005monotonicity}. Domain monotonicity for the first  Robin eigenvalue of the $p$-Laplacian for particular cases were established in \cite{mallick2022shape} and \cite{gavitone2018first}. We use the shape derivative formula obtained above to study the domain monotonicity property of the first eigenvalue of \eqref{Robin p-Laplace evp with potential term}.
If $V(x)\leq \lambda_1 \; a.e. $ on $\partial \Omega,$ we show that domain monotonicity holds for the ball case for all values of the boundary parameter $\beta$ and also for more general smooth bounded domains in the case of a large boundary parameter $\beta$. As an application of the differentiability of the first eigenvalue with respect to $\beta$ as discussed above, we show that it satisfies the domain monotonicity property when the domain is scaled.  

The paper is organized as follows. In the next section we prove properties of the first eigenvalue. In section 3, we study the differentiability with respect to the boundary parameter $\beta$. In the last section, we obtain shape derivative formula for the first eigenvalue. 


\section{Preliminaries}
	\label{Preliminaries section}
In this section, we mostly follow the presentation in \cite{bandle2015second}. Let $\Omega \subset \mathbb{R}^n$ be a bounded domain in the H$\ddot{\text{o}}$lder class $C^{2,\alpha}.$ Let $P$ be a point on $\partial \Omega$ and $T_P$ denote the tangent space of $P.$ Then there is a neighborhood $\Omega_P$ of $P$ and a Cartesian coordinate system having an orthonormal basis $B_P:=\{e_i\}, i=1,\ldots,n, $ centered at $P,$ with $e_i \in T_P, i=1,\ldots,n-1, $ and $e_n$ in the direction of unit outer normal $\eta$. Let $(\xi_1, \xi_2,\ldots,\xi_n)$ denote the coordinates with respect to $B_P.$ Also, we make the assumption $ \Omega \cap \Omega_P =\{\xi \in \Omega_P : \xi_n <F(\xi_1, \xi_2,\ldots,\xi_{n-1}) \}, F \in C^{2,\alpha}. $ From this, it is clear that $F_{\xi_i}(0)=0$ for $i=1,2,\ldots,n-1$. Let $\xi ' := (\xi_1, \xi_2,\ldots,\xi_{n-1}).$ In the $\xi ' -$ coordinate system,  $x(\xi ') = (\xi_1, \xi_2,\ldots,\xi_{n-1}, F(\xi '))$ represents points on $\Omega_P \cap  \partial \Omega $ and $\tilde \eta  (\xi ')= (\tilde \eta_1,\tilde \eta_2,\ldots,\tilde \eta_n) $ the unit outer normal. We will follow the Einstein convention, where repeated indices indicate summation from $1$ to $n-1$ or from $1$ to $n$. For $i=1,2,\ldots,n-1,$ the vectors $x_{\xi_i}:=\frac{\partial x}{\partial \xi_i}$ span $T_P.$ Let $g_{ij}:=x_{\xi_i} \cdot x_{\xi_j}$ be the metric tensor of $ \partial \Omega,$ and its inverse be denoted by $g^{ij}.$

Note that any vector $v$ can be written as $v=v^\tau +(v \cdot \eta)\eta,$ where $v^\tau := g^{ij} (v \cdot x_{\xi_i}) x_{\xi_j} $ is the projection to the tangent space. Let $v:\partial \Omega \to \mathbb{R}^n$ be a smooth vector field. The tangential divergence of $v$ is defined as 
\begin{equation} 
\label{Equation to be used in sigma A}
\text{div}_{\partial \Omega} v=g^{ij} \tilde{v}_{\xi_i} \cdot x_{\xi_j}, 
\end{equation} where $ \tilde{v}=v(\xi',F(\xi')).$ For any $f \in C^1(\Omega_P),$ the tangential gradient of $f|_{\partial\Omega}$ is given by
\begin{equation}
    \label{grad tau f on boundary}
    \nabla ^\tau f =g^{ij} (\nabla f \cdot x_{\xi_i}) x_{\xi_j}.
\end{equation}

For $f\in C^1(\partial \Omega),$ $v\in C^{0,1}(\partial \Omega , \mathbb{R}^n ),$ and mean curvature $H$ of the boundary $\partial \Omega,$ the Gauss theorem on surfaces can be stated as 
\begin{equation}
\label{Gauss theorem on surfaces}
    \int_{\partial \Omega}f \text{ div}_{\;\partial \Omega} \;v~ds = - \int_{\partial \Omega}v \cdot \nabla ^\tau f ~ds + (n-1) \int_{\partial \Omega}f (v\cdot \eta) H ~ds.
\end{equation}

\subsection{Hadamard's method of variations}
In this technique, we consider a family of diffeomorphisms
$T_t\in \mathcal{C}^1(\Omega, \mathbb{R}^n)$,  $|t|\in \mathbb{R}$ sufficiently small. We define
\begin{equation*}\label{T_t}
    T_t(x)=x+tv(x)+o(t),
\end{equation*}where  $v=(v_1(x),v_2(x),..., v_n(x)) $ is a $C^1(\overline{\Omega})$ vector field and $o(t)$ denotes all the terms with $\frac{o(t)}{t}\to 0$ as $t\to 0.$
Now letting {\begin{equation}\label{Omega_t}
    \Omega_t=T_t(\Omega),
\end{equation}the perturbation of energy functionals can be studied by analyzing the dependence of the functionals on $\Omega_t$ with respect to $t$. 

\subsection{Perturbation of domain}
\label{Perturbation of domain section}
For a $C^{2,\alpha}$ vector field $v$, let $\Omega_t$ be a family of perturbed as defined in \eqref{Omega_t} . The volume element of $\Omega_t$ and surface element of $\partial \Omega_t$ have been computed in \cite{bandle2015second}. We include the same here for the sake of completeness. 

\noindent Let $y(t,\Omega):=y.$ The Jacobian matrix (up to first order terms) of $y$ is given by $I +t D_v,$ where $(D_v)_{ij}=\frac{\partial v_i}{\partial x_j}.$ When $|t|$ is small, using Jacobi's formula, we get 
\begin{equation}
\label{J(t)}
J(t):=\text{det} (I+tD_v) =1+t \text{ div } v+ o(t).
\end{equation}
Now $|\Omega_t|=\int_\Omega J(t) ~dx = |\Omega|+ t \int_\Omega \text{ div } v ~dx +  o(t).$

 In order to compute the surface element of the boundary, $\partial \Omega_t $ can be represented locally as $ \{ y(\xi ') := x(\xi ')+t \tilde v(\xi ') : \xi ' \in \Omega_P \cap \{\xi_n=0\} \}.$ 
Let $a_{ij}:=x_{\xi_i} \cdot \tilde v(\xi_j) + x_{\xi_j} \cdot \tilde v(\xi_i), b_{ij}:=2\tilde v(\xi_i) \cdot \tilde v(\xi_j)$ and $g_{ij}^{t}:=g_{ij}+ta_{ij}+\frac{t^2}{2}b_{ij}.$
Then $|dy|^2=\left( g_{ij}+ta_{ij}+\frac{t^2}{2}b_{ij} \right) d\xi_i d\xi_j=g_{ij}^{t}d\xi_i d\xi_j.$
Now let $G:=(g_{ij}), G^{-1}:=(g^{ij}), A:=(a_{ij}), B:=(b_{ij}) $ and $ G^t:=(g_{ij}^{t}).$ The surface element on $\partial \Omega_t$ is given by $\sqrt{\text{det } G^t}d\xi '.$
Then $\sqrt{\text{det } G^t}=\left(\sqrt{\text{det } G}\right) k(x,t)^{1/2}, $ where $k(x,t):= \text{det }\left( I+tG^{-1}A+\frac{t^2}{2}G^{-1}B \right) .$ Let $\sigma_A:= \text{trace }G^{-1}A,  \sigma_B:= \text{trace }G^{-1}B$ and $ \sigma_{A^2}:= \text{trace }(G^{-1}A)^2.$
Using the Taylor expansion, for small $|t|,$ we have
$k(x,t)=1+t\sigma_A + o(t). $

Let $m(t):=\sqrt{k(x,t)}=1+\frac{t}{2}\sigma_A +o(t).$ Then the surface element $ds_t$ of $\partial \Omega_t$ can be written as $ds_t=m(t) ds,$ where $ds$ is the surface element of $\partial \Omega.$

 Using \eqref{Equation to be used in sigma A}, we have $\sigma_A=2 g^{ij} \tilde v_{\xi_j} \cdot x_{\xi_i} = 2 \text{div}_{\partial\Omega} v.$ It can be shown that $\frac{1}{2}\sigma_A= \text{div}_{\partial\Omega} v^\tau +(n-1)H(v \cdot \eta),$ where $H$ denotes the mean curvature of the boundary. So we have \begin{equation}
\label{m dot 0}
\dot m (0)=\frac{1}{2}\sigma_A= \text{div}_{\partial\Omega} v^\tau +(n-1)H(v \cdot \eta).
\end{equation}

We also include the following Lemma from \cite{mallick2022shape} for the sake of completeness.
\begin{lemma} \cite{mallick2022shape}
	\label{Lemma for Aij}
	Let  $y=x+tv(x)+o(t)$  
 and $A_{ij}(t):=\frac{\partial x_i}{\partial y_k}\frac{\partial x_j}{\partial y_k} \left(J(t)\right)^\frac{2}{p},$ with Jacobian $J(t).$ We have \begin{enumerate}
	    \item[(i)] $J(0)=1$ and $\dot{J}(0)=div\; v,$
	    \item[(ii)] $A_{ij}(0)= \delta_{ij}$ and $\dot{A_{ij}}(0)= \frac{2}{p}\delta_{ij}\;div\; v-\partial_jv_i-\partial_iv_j,$ where $\partial_jv_i:=\frac{\partial v_i}{\partial x_j}.$
	\end{enumerate}
\end{lemma}

\section{Properties of the first eigenvalue}
\begin{proof}
    (i) and (ii) follow from the characterization of $\lambda_1(V).$

    Consider $0<C<\frac{\lambda_1(V)}{\lambda_1(V)+\|V\|_{\infty}},$ that is, $C\|V\|_{\infty}\leq (1-C)\lambda_1(V).$ Then for all $u\in W^{1,p}(\Omega),$ using the definition of $\lambda_1(V)$, we get 
\begin{align*}
    (1-C)\int_{\Omega}(|\nabla u|^p+V|u|^p)~dx+(1-C)\beta\int_{\partial\Omega}|u|^p~ds&\geq(1-C)\lambda_1(V)\int_{\Omega}|u|^p~dx\\
    &\geq C\|V\|_{\infty}\int_{\Omega}|u|^p~dx\\
    &\geq-C\int_{\Omega}V|u|^p~dx.
\end{align*}
Since $\beta>0$, we have 
$$(1-C)\int_{\Omega}\left(|\nabla u|^p+V|u|^p\right)~dx+\beta\int_{\partial\Omega}|u|^p~ds\geq -C\int_{\Omega}V|u|^p~dx.$$
The above equation can be rearranged as  
\begin{equation}\label{potential-Bounded below by C}
C\int_{\Omega}|\nabla u|^p~dx\leq \int_{\Omega}|\nabla u|^p~dx+\int_{\Omega}V|u|^p~dx+\int_{\partial\Omega} \beta |u|^p ~ds,\quad \forall u\in{W^{1,p}(\Omega)}.
\end{equation}

\noindent The following inequality can be obtained using the variational characterization of $\lambda_1(V).$ \begin{equation}
    \label{potential-Bounded below by sigma_1}
     \lambda_1(V) \int_{\Omega} |u|^{p} ~dx \leq \int_{\Omega} |\nabla u |^{p} ~dx + \int_{\Omega} V |u|^{p} ~dx + \int_{\partial \Omega} \beta |u|^{p} ~ds,  \quad \forall u\in{W^{1,p}(\Omega)}.
\end{equation}
Let $M:= \frac{1}{2}\displaystyle{\inf} \{ C, \lambda_1(V) \}$.
Adding \eqref{potential-Bounded below by C} and \eqref{potential-Bounded below by sigma_1}, we arrive at 
\begin{equation*}
    M  \|u \|^p_{W^{1,p}(\Omega)} \leq \int_{\Omega}|\nabla u|^p ~dx+\int_{\Omega}V|u|^p ~dx+
    \int_{\partial\Omega}\beta|u|^p ~ds,  \quad \forall u\in W^{1,p}(\Omega).
\end{equation*}

\end{proof}


\section{Monotonicity of $\lambda_1(V)$ with respect to the Robin boundary parameter}

Let $u_\beta$ be the first eigenfunction corresponding to $ \lambda_1(V)$ with the normalization $\|u_\beta\|_{L^\infty} = 1$ for any $\beta>0.$
\begin{equation*}
   \lambda_1(\beta)= \displaystyle{\inf_{ u \in W^{1,p}(\Omega) }} \frac{\int_{\Omega} |\nabla u |^{p} ~dx + \int_{\Omega} V(x)|u|^{p} ~dx + \int_{\partial \Omega} \beta |u|^{p} ~ds }{ \int_{\Omega} |u|^{p} ~dx} 
\end{equation*}
Let $\beta_1 > \beta.$ Consider
\begin{equation*}
\begin{aligned}
    \lambda_1(\beta_1)- \lambda_1(\beta) &=  \lambda_1(\beta_1)- \displaystyle{\inf_{ u \in W^{1,p}(\Omega) }} \frac{\int_{\Omega} |\nabla u |^{p} ~dx + \int_{\Omega} V(x)|u|^{p} ~dx + \int_{\partial \Omega} \beta |u|^{p} ~ds }{ \int_{\Omega} |u|^{p} ~dx} \\
    &\geq \lambda_1(\beta_1)-  \frac{\int_{\Omega} |\nabla u_{\beta_1} |^{p} ~dx + \int_{\Omega} V(x)|u_{\beta_1}|^{p} ~dx + \int_{\partial \Omega} \beta |u_{\beta_1}|^{p} ~ds }{ \int_{\Omega} |u_{\beta_1}|^{p} ~dx} \\
    &= (\beta_1 - \beta ) \frac{\int_{\partial \Omega} |u_{\beta_1}|^{p} ~ds }{ \int_{\Omega} |u_{\beta_1}|^{p} ~dx}.
\end{aligned}
\end{equation*}
Now 
\begin{equation*}
\begin{aligned}
    \lambda_1(\beta_1)- \lambda_1(\beta) &= \displaystyle{\inf_{ u \in W^{1,p}(\Omega) }} \frac{\int_{\Omega} |\nabla u |^{p} ~dx + \int_{\Omega} V(x)|u|^{p} ~dx + \int_{\partial \Omega} \beta_1 |u|^{p} ~ds }{ \int_{\Omega} |u|^{p} ~dx} -  \lambda_1(\beta)\\
    &\leq \frac{\int_{\Omega} |\nabla u_{\beta} |^{p} ~dx + \int_{\Omega} V(x)|u_{\beta}|^{p} ~dx + \int_{\partial \Omega} \beta_1 |u_{\beta}|^{p} ~ds }{ \int_{\Omega} |u_{\beta}|^{p} ~dx} -  \lambda_1(\beta)\\
    &= (\beta_1 - \beta ) \frac{\int_{\partial \Omega} |u_{\beta}|^{p} ~ds }{ \int_{\Omega} |u_{\beta}|^{p} ~dx}.
\end{aligned}
\end{equation*}
So 
\begin{equation*}
    \frac{\int_{\partial \Omega} |u_{\beta_1}|^{p} ~ds }{ \int_{\Omega} |u_{\beta_1}|^{p} ~dx} \leq \frac{ \lambda_1(\beta_1)- \lambda_1(\beta) }{\beta_1-\beta} \leq \frac{\int_{\partial \Omega} |u_{\beta}|^{p} ~ds }{ \int_{\Omega} |u_{\beta}|^{p} ~dx}.
\end{equation*}
 From the characterization of the first eigenvalue, it follows that the first Robin eigenvalue $\lambda_1^\beta(V) \leq \lambda_1^D(V),$ the first Dirichlet eigenvalue. As $V \in C^1(\mathbb{R}^n),$ it is bounded in $\overline{\Omega}.$ Also, since  $\|u_\beta\|_{L^\infty} = 1$ for any $\beta>0$, the term $\int_{\Omega} V(x)|u_\beta|^{p} ~dx$ in the characterization of $\lambda_1^\beta(V)$ can be shown to be uniformly bounded below. As all other terms in the characterization of $\lambda_1^\beta(V)$ are nonnegative, it follows that $\lambda_1^\beta(V)$ is uniformly bounded below.  Thus, $(\lambda_1-V) |u_\beta|^{p-2}u_\beta$
	 is uniformly bounded in $L^\infty(\Omega)$.   Therefore, using 
	 \cite[Theorem 2]{lieberman1988boundary} we have $\{\|u_\beta\|_{C^{1,\alpha}(\overline{\Omega})}\}$ is bounded and hence, by the Arzela-Ascoli theorem,  $u_{\beta_1} \to u_{\beta}$ (upto a subsequence) in $C^1(\overline{\Omega})$ as $\beta_1 \to \beta .$ 
 
Therefore, we have
\begin{equation*}
    \frac{d\lambda_1}{d\beta} = \displaystyle{\lim_{\beta_1 \to \beta}} \frac{ \lambda_1(\beta_1)- \lambda_1(\beta) }{\beta_1-\beta} = \frac{\int_{\partial \Omega} |u_{\beta}|^{p} ~ds }{ \int_{\Omega} |u_{\beta}|^{p} ~dx}.
\end{equation*}

\begin{remark}
\begin{itemize}
    \item [(i)]Note that $\frac{d\lambda_1}{d\beta} >0.$ So the first eigenvalues are strictly increasing with respect to $\beta.$ Also, the first Robin eigenvalues $\lambda_1^\beta(V)$ are bounded by the  first Dirichlet eigenvalue $\lambda_1^D(V).$
     \item [(ii)] In Proposition 5.1 of \cite{mallick2022shape}, we proved that for $V \equiv 0, $ $\lambda_1^\beta \to \lambda_1^D $ and the corresponding eigefunctions $u_\beta \to u_D$ as $\beta \to \infty.$ We can generalize that for $\lambda_1(V)$ for $V \in C^1(\mathbb{R}^n),$ and the proof follows by a gentle modification of the proof of Proposition 5.1 of \cite{mallick2022shape}.
\end{itemize}
\end{remark}
\begin{proposition}
    Let $\Omega$ be a $C^{2,\alpha}$ bounded domain in $\mathbb{R}^n$ and let $V\equiv0.$ For $t \geq 1, $ and $\beta$ positive, $\lambda_1$ satisfies the domain monotonicity property $\lambda_1 (t \Omega,\beta) \leq \lambda_1 (\Omega,\beta).$
\end{proposition}

\begin{proof}
    After rescaling the Rayleigh quotient, we get 
\begin{equation*}\label{Rescaling Rayleigh quotient-1}
     \lambda_{1}\left(t \Omega, \beta /|t|^{p-1}\right)=\frac{1}{|t|^{p}}  \lambda_{1}(\Omega, \beta).
\end{equation*}

\noindent Therefore, we have \begin{equation}\label{Rescaling Rayleigh quotient-2}
     \lambda_{1}\left(t \Omega, \beta \right)=\frac{1}{|t|^{p}}  \lambda_{1}(\Omega, |t|^{p-1}\beta).
\end{equation}

\noindent Define $f(\beta):=\frac{\lambda_{1}(\Omega, \beta)}{\beta}.$ We will show that $f$ is decreasing.

$$
\begin{aligned}
f^{\prime}(\beta) & =\frac{\lambda_{1}^{\prime}(\Omega, \beta) \beta-\lambda_{1}(\Omega, \beta)}{\beta^{2}} \\
& =\frac{1}{\beta^{2}}\left \{\frac{\int_{\partial \Omega}\left|u_{\beta}\right|^{p} ds}{\int_{\Omega}\left|u_{\beta}\right|^{p} d x}\beta -\frac{\int_{\Omega}\left|u_{\beta}\right|^{p} d x+\beta \int_{\partial \Omega}\left|u_{\beta}\right|^{p} ds}{\int_{\Omega}\left|u_{\beta}\right|^{p} d x} \right \}\\
& <0.
\end{aligned}
$$

\noindent So $$\frac{\lambda_{1}(\Omega, \beta)}{\beta} \geq \frac{\lambda_{1}\left(\Omega,|t|^{p-1} \beta\right)}{|t|^{p-1} \beta}.$$
This implies $
|t|^{p-1} \lambda_{1}(\Omega, \beta) \geq \lambda_{1}\left(\Omega,|t|^{p-1} \beta\right).
$ Multiplying both sides by $\frac{1}{|t|^{p}},$ we get 
\begin{equation}\label{Inequality using a decreasing function f}
\frac{1}{|t|} \lambda_{1}(\Omega, \beta) \geq \frac{1}{|t|^{p}} \lambda_{1}\left(\Omega,|t|^{p-1} \beta\right). 
\end{equation}

\noindent Using \eqref{Inequality using a decreasing function f} in \eqref{Rescaling Rayleigh quotient-2}, 

$$
\lambda_{1}(t \Omega, \beta) \leq \frac{1}{|t|} \lambda_{1}(\Omega, \beta) \leq \lambda_{1}(\Omega, \beta) \text { as } t \geq 1.
$$

\end{proof}

\section{Shape derivative formula}
For a fixed potential $V \in C^1(\mathbb{R}^n),$  we obtain a shape derivative formula for the first eigenvalue as discussed below and use it to study the domain monotonicity property for the same.

Let us consider the energy functional 
\begin{equation}
\label{First energy functional}
E_{\Omega_t}(u):=\int_{\Omega_t} |\nabla_y u |^p dy +\int_{\Omega_t}V(y) |u|^p dy -\int_{\Omega_t}\lambda(\Omega_t) |u|^p dy + \beta \int_{\partial \Omega_t}  |u|^p ~ds_t.
\end{equation}

\noindent Assume $\tilde{u} \in W^{1,p}(\Omega_t)$ is a critical point of (\ref{First energy functional}).Therefore, it satisfies the following Euler-Lagrange equation 
	\begin{equation}
\label{Euler Lagrange equation in Omega_t}
\left\{
\begin{split}
-\Delta_p \tilde{u}+ V(y)|\tilde{u}|^{p-2}\tilde{u}&=\lambda(\Omega_t)  |\tilde{u}|^{p-2}\tilde{u} \quad \mathrm{in} ~\Omega_t ,\\
|\nabla \tilde{u}|^{p-2}\frac{\partial \tilde{u}}{\partial\eta_t}+\beta|\tilde{u}|^{p-2}\tilde{u}&
=0\quad\mathrm{on}~\partial\Omega_t,
\end{split}
\right.
\end{equation} where $\eta_t$ is the outer unit normal to the boundary $\partial \Omega_t.$

\noindent Define $E(t):=E_{\Omega_t}(\tilde{u})$ and $\lambda(t):=\lambda(\Omega_t).$ Let  $y=x+tv(x)+o(t),$ where $x(y)$ is its inverse. Then by a change of variables, we have
\begin{equation*}
\label{E in Omega}
\begin{aligned}
E(t)&=\int_{\Omega} \left[(\nabla \tilde{u}(t))^T A(t) \nabla \tilde{u}(t) \right]^\frac{p}{2} ~dx +\int_{\Omega}\tilde{V} |\tilde{u}(t)|^pJ(t)dx  -\int_{\Omega}\lambda(t) |\tilde{u}(t)|^pJ(t)~dx \\
&+ \beta \int_{\partial \Omega}  |\tilde{u}(t)|^p m(t)~ds,
\end{aligned}
\end{equation*}
where $A_{ij}(t):=\frac{\partial x_i}{\partial y_k}\frac{\partial x_j}{\partial y_k} \left(J(t)\right)^\frac{2}{p},$ with Jacobian $J(t)$ as given in \eqref{J(t)},  $\tilde{u}(t):=\tilde{u}(x+tv(x),t), $ for $t \in (-\epsilon, \epsilon)$ with $\epsilon > 0$ small enough, and $\tilde{V}:=V(x+tv+o(t)).$
	
\noindent Now, $\tilde{u}(t)$ solves the corresponding Euler-Lagrange equation in $\Omega$ as given below.  
 \begin{equation}
 \label{Euler Lagrange equation in Omega}
 \left\{
 \begin{split}
 -L_A\tilde{u}(t)+ \tilde{V}  |\tilde{u}(t)|^{p-2}\tilde{u}(t)J(t)&=\lambda(t)  |\tilde{u}(t)|^{p-2}\tilde{u}(t)J(t)\quad \mathrm{in} ~\Omega ,\\
 \partial_{\eta_A}\tilde{u}(t)+\beta|\tilde{u}(t)|^{p-2}\tilde{u}(t)m(t)&=0\quad\mathrm{on}~\partial\Omega,
 \end{split}
 \right.
 \end{equation}
 where $L_A(\tilde{u}(t)):=\frac{\partial}{\partial x_j}\{\left[(\nabla \tilde{u}(t))^T A \nabla \tilde{u}(t) \right]^\frac{p-2}{2}A_{ij} \tilde{u}_{x_i}(t) \}$ and \\ $\partial_ {\eta_A}(\tilde{u}(t)) :=\left[(\nabla \tilde{u}(t))^T A \nabla \tilde{u}(t) \right]^\frac{p-2}{2}A_{ij} \tilde{u}_{x_i}(t) \eta^j. $
 
 \noindent The equation \eqref{Euler Lagrange equation in Omega} can be written in the weak form as  
 \begin{equation}
\label{Weak form of Euler Lagrange equation in Omega}
\begin{aligned}
    \int_{\Omega} \left[(\nabla \tilde{u})^T A \nabla \tilde{u} \right]^\frac{p-2}{2} 
    &\left[(\nabla \tilde{u})^T A \nabla \phi \right]~dx  
    +  \int_{\Omega} \tilde{V}  |\tilde{u}|^{p-2}\tilde{u}J(t)~dx \\
    &+ \beta \int_{\partial \Omega}  |\tilde{u}|^{p-2} \tilde{u} m \phi ~ds 
    = \int_{\Omega}\lambda(t) |\tilde{u}|^{p-2}\tilde{u}J\phi ~dx.
\end{aligned}
\end{equation}
for all $\phi \in W^{1,p}(\Omega).$

 \textcolor{black}{It is possible to prove that $\tilde{u}(t) \to \tilde{u}(0)$ in $C^1(\overline{\Omega})$ using similar arguments as in Theorem 1 of \cite{melian2001perturbation}.
 Therefore, $\tilde{u}(t)$ can be expanded as $\tilde{u}(t)=\tilde{u}(0)+t\dot{\tilde{u}}(0)+o(t).$} Also $\tilde{u}(0):=u(x).$ $A_{ij}(t) $ can be written as $A_{ij}(t)=A_{ij}(0)+t\dot{A_{ij}}(0)+o(t).$ Hence Lemma \ref{Lemma for Aij} holds. 

	
Since the first eigenvalue $\lambda_1(t):=\lambda_1(\Omega_t)$ of (\ref{Euler Lagrange equation in Omega_t}) is given by the following Rayleigh quotient $$\lambda_1(t)=\frac{\int_{\Omega_t} |\nabla_y \tilde{u} |^p dy + \int_{\Omega_t}V(y) |\tilde{u}|^p dy+ \beta \int_{\partial \Omega_t}  |\tilde{u}|^p ~ds_t}  {\int_{\Omega_t}  |\tilde{u}|^p dy}, $$
we can employ a change of variables to get 
\begin{equation}
\label{Rayleigh quotient after change of variables}
\lambda_1(t)=\frac{\int_{\Omega}  \left[(\nabla \tilde{u}(t))^T A(t) \nabla \tilde{u}(t) \right]^\frac{p}{2} ~dx + \int_{\Omega} \tilde{V} |\tilde{u}|^p J(t) dx  + \beta \int_{\partial \Omega}  |\tilde{u}(t)|^p m(t) ~ds}  {\int_{\Omega}  |\tilde{u}(t)|^p J(t) ~dx}. 
\end{equation} 

\noindent Note that \eqref{Euler Lagrange equation in Omega_t} is $(p-1)$ homogeneous.   Therefore, we apply the normalization $\int_{\Omega}  |\tilde{u}(t)|^p J(t) ~dx=1$ to arrive at the following equation.   
\begin{equation}
\label{Derivative of normalization}
\frac{d}{dt}\left( \int_{\Omega}  |\tilde{u}(t)|^p J(t) ~dx \right) = p\int_{\Omega}  |\tilde{u}(t)|^{p-2} \tilde{u}(t) \dot{\tilde{u}}(t)J(t) ~dx +\int_{\Omega}  |\tilde{u}(t)|^p \dot{J}(t) ~dx=0.
\end{equation}

\noindent In \eqref{Weak form of Euler Lagrange equation in Omega}, we select $\dot{\tilde{u}} $ as a test function, and obtain
\begin{equation}
\label{After putting test function in weak form}
\begin{aligned}
\int_{\Omega} \left[(\nabla \tilde{u})^T A \nabla \tilde{u} \right]^\frac{p-2}{2} &\left[(\nabla \tilde{u})^T A \nabla \dot{\tilde{u}} \right]~dx  + \int_{\Omega} \tilde{V}  |\tilde{u}|^{p-2}\tilde{u}\dot{\tilde{u}} J(t) ~dx \\
&+ \beta \int_{\partial \Omega}  |\tilde{u}|^{p-2} \tilde{u} \dot{\tilde{u}} m  ~ds
=\int_{\Omega}\lambda_1(t) |\tilde{u}|^{p-2}\tilde{u} \dot{\tilde{u}} J ~dx.
\end{aligned}
\end{equation}

\noindent Now, by multiplying by $p$ on both sides of (\ref{After putting test function in weak form}) and using (\ref{Derivative of normalization}), we have
\begin{equation}
\label{For substituting in lambda dot}
\begin{aligned}
p\int_{\Omega} \left[(\nabla \tilde{u})^T A \nabla \tilde{u} \right]^\frac{p-2}{2} &\left[(\nabla \tilde{u})^T A \nabla \dot{\tilde{u}} \right]~dx + p\int_{\Omega} \tilde{V}  |\tilde{u}|^{p-2}\tilde{u}\dot{\tilde{u}} J(t) ~dx \\
&+ p\beta \int_{\partial \Omega}  |\tilde{u}|^{p-2} \tilde{u} \dot{\tilde{u}} m  ~ds =-\lambda_1(t)\int_{\Omega}  |\tilde{u}|^p \dot{J}~dx.
\end{aligned}
\end{equation}
Differentiating both sides of (\ref{Rayleigh quotient after change of variables}) with respect to $t$ and substituting from (\ref{For substituting in lambda dot}), we get
\begin{equation*}
\label{lambda dot t}
\begin{aligned}
\dot{\lambda}_1(t)&=\frac{p}{2}\int_{\Omega} \left[(\nabla \tilde{u})^T A \nabla \tilde{u} \right]^\frac{p-2}{2} \left[(\nabla \tilde{u})^T \dot{A} \nabla {\tilde{u}} \right]~dx -\lambda_1(t)\int_{\Omega}  |\tilde{u}|^p \dot{J}~dx+ \beta \int_{\partial \Omega}  |\tilde{u}|^p \dot{m} ~ds  \\
&+ \int_{\Omega}\dot{\tilde{V}} v |\tilde{u}|^p J ~dx + \int_{\Omega}\tilde{V} |\tilde{u}|^p \dot{J} ~dx.
\end{aligned}
\end{equation*} 

\noindent Thus 
\begin{equation*}
\label{lambda dot 0 starting equation}
\begin{aligned}
    \dot{\lambda}_1(0)&=\frac{p}{2}\int_{\Omega} \left[(\nabla {u})^T A(0) \nabla {u} \right]^\frac{p-2}{2} \left[(\nabla {u})^T \dot{A}(0) \nabla {u} \right]~dx -\lambda_1(0)\int_{\Omega}  |u|^p \dot{J}(0)~dx\\
    &+ \beta \int_{\partial \Omega}  |u|^p \dot{m}(0) ~ds 
     + \int_{\Omega}\dot{\tilde{V}} v |\tilde{u}|^p J(0) ~dx + \int_{\Omega}\tilde{V} |\tilde{u}|^p \dot{J}(0) ~dx .
\end{aligned}
\end{equation*} 
Define $\dot{\Lambda}_1:=\frac{p}{2}\int_{\Omega} \left[(\nabla {u})^T A(0) \nabla {u} \right]^\frac{p-2}{2} \left[(\nabla {u})^T \dot{A}(0) \nabla {u} \right]~dx, \; \dot{\Lambda}_2:=-\lambda_1(0)\int_{\Omega}  |u|^p \dot{J}(0)~dx,$ $\dot{\Lambda}_3:=\beta \int_{\partial \Omega}  |u|^p \dot{m}(0) ~ds,$ $\dot{\Lambda}_4:=\int_{\Omega}\dot{\tilde{V}} v |\tilde{u}|^p J(0) ~dx,$ and $\dot{\Lambda}_5:=\int_{\Omega}\tilde{V} |\tilde{u}|^p \dot{J}(0) ~dx.$

\noindent Then $$ \dot{\Lambda}_1= \frac{p}{2}\int_{\Omega} \left[u_{x_i} A_{ij}(0) u_{x_j} \right]^\frac{p-2}{2} \left[u_{x_i} \dot{A}_{ij}(0) u_{x_j} \right] ~dx. $$

\noindent Now, applying Lemma \ref{Lemma for Aij} and simplifying, we have
\begin{equation}
\label{la_1 dash}
   \dot{\Lambda}_1=\int_{\Omega} |\nabla {u}|^{p} \text{ div } v ~dx -p\int_{\Omega} |\nabla {u}|^{p-2}\partial_j v_i u_{x_i} u_{x_j} ~dx.
\end{equation}


Note that the eigenfunction $u$ has only $C^{1,\alpha}$ regularity. It would have been possible to obtain a simplified formulation for the above equation using integration by parts if $u$ had $C^2$ regularity. As an alternative, we introduce the following perturbed problem for small $\varepsilon>0$, 
adopting the methodology in \cite{melian2001perturbation} and \cite{del2009some}. 

\begin{equation}
	\label{Epsilon Robin p-Laplacian}
	\left\{
	\begin{split}
-\mbox{div} ((|\nabla u_\varepsilon|^2+\varepsilon ^2)^\frac{p-2}{2} \nabla u_\varepsilon)+V|u_\varepsilon|^{p-2}u_\varepsilon &= \lambda_\varepsilon  |u_\varepsilon|^{p-2}u_\varepsilon \quad \mathrm{in} ~\Omega,\\
	-(|\nabla u_\varepsilon|^2+\varepsilon ^2)^\frac{p-2}{2}\frac{\partial u_\varepsilon}{\partial\eta}&=\beta|u_\varepsilon|^{p-2}u_\varepsilon \quad\mathrm{on}~\partial\Omega.
	\end{split}
	\right.
	\end{equation} 
In this case, $u_\varepsilon$ belongs to   $C^{2,\theta}(\overline{\Omega})$ for some $\theta \in (0,1).$ Also, $u_\varepsilon \to u$ in $C^1(\overline{\Omega})$ as $\varepsilon \to 0$ (see \cite{MR0241822} and \cite{lieberman1988boundary}).

\noindent Let 
\begin{equation}
\label{Lambda_1 epsilon dot}
   \dot{\Lambda}_1^\varepsilon=\int_{\Omega} (|\nabla u_\varepsilon|^2+\varepsilon ^2)^\frac{p}{2} \text{ div } v ~dx -p\int_{\Omega}(|\nabla u_\varepsilon|^2+\varepsilon ^2)^\frac{p-2}{2}\partial_j v_i (u_\varepsilon)_{x_i} (u_\varepsilon)_{x_j} ~dx.
\end{equation}

\noindent Keep in mind that 
\begin{equation}
    \label{p-Laplace expansion for perturbed problem}
    \begin{aligned}
    \mbox{div} ((|\nabla u_\varepsilon|^2+\varepsilon ^2)^\frac{p-2}{2} \nabla u_\varepsilon)&=(|\nabla u_\varepsilon|^2+\varepsilon ^2)^\frac{p-2}{2} \Delta u_\varepsilon \\
    &+ (p-2) (|\nabla u_\varepsilon|^2+\varepsilon ^2)^\frac{p-4}{2} (u_\varepsilon)_{x_i} (u_\varepsilon)_{x_j x_i} (u_\varepsilon)_{x_j}.
    \end{aligned}
\end{equation}
\noindent Using the integration by parts in \eqref{Lambda_1 epsilon dot} and substituting from \eqref{p-Laplace expansion for perturbed problem}, we arrive at 
 
\begin{equation}
 \label{Lambda_1 epsilon dot new}
    \begin{aligned}
     \dot{\Lambda}_1^\varepsilon &= p\int_{\Omega} \mbox{div} ((|\nabla u_\varepsilon|^2+\varepsilon ^2)^\frac{p-2}{2} \nabla u_\varepsilon) (\nabla u_\varepsilon \cdot v)~dx \\
     &-p\int_{\partial \Omega}(|\nabla u_\varepsilon|^2+\varepsilon ^2)^\frac{p-2}{2}(\nabla u_\varepsilon \cdot v)(\nabla u_\varepsilon \cdot \eta) ~ds+\int_{\partial \Omega} (|\nabla u_\varepsilon|^2+\varepsilon ^2)^\frac{p}{2}(v \cdot \eta)~ds.
   \end{aligned}
 \end{equation}

\noindent Applying the integration by parts in $\dot{\Lambda}_2,$ we get $$\dot{\Lambda}_2=  \lambda_1 (0) \int_{\Omega} p |u|^{p-2} u (\nabla u \cdot v) ~dx - \lambda_1 (0)  \int_{\partial \Omega} |u|^p (v \cdot \eta) ~ds. $$
\noindent Consider  $$\dot{\Lambda}_2^\varepsilon=  \lambda_\varepsilon \int_{\Omega} p |u_\varepsilon|^{p-2} u_\varepsilon (\nabla u_\varepsilon \cdot v) ~dx - \lambda_1 (0)  \int_{\partial \Omega} |u|^p (v \cdot \eta) ~ds.$$
As $\lambda_\varepsilon |u_\varepsilon|^{p-2} u_\varepsilon = - \mbox{div} ((|\nabla u_\varepsilon|^2+\varepsilon ^2)^\frac{p-2}{2} \nabla u_\varepsilon) + V|u_\varepsilon|^{p-2}u_\varepsilon$ in $\Omega ,$  $\dot{\Lambda}_2^\varepsilon$ can be written as 
\begin{equation}
 \label{Lambda_2 epsilon dot}
 \begin{aligned}
    \dot{\Lambda}_2^\varepsilon &= - p\int_{\Omega} \mbox{div} ((|\nabla u_\varepsilon|^2+\varepsilon ^2)^\frac{p-2}{2} \nabla u_\varepsilon) (\nabla u_\varepsilon \cdot v)~dx\\
    &+ p \int_{\Omega} V|u_\varepsilon|^{p-2}u_\varepsilon  (\nabla u_\varepsilon \cdot v)~dx 
    -\lambda_1 (0)  \int_{\partial \Omega} |u|^p (v \cdot \eta) ~ds.
    \end{aligned}
\end{equation}

\noindent Now, using (\ref{m dot 0}) in $\dot{\Lambda}_3,$  we have 
\begin{equation}
  \label{Lambda_3 dot}
 \dot{\Lambda}_3= \beta \int_{\partial \Omega} |u|^p \{ \text{div}_{\partial\Omega} v^\tau +(n-1)H(v \cdot \eta) \} ~ds.  
\end{equation}

\noindent Applying Lemma \ref{Lemma for Aij} and using the integration by parts, we get 
\begin{equation}\label{Lambda_5 dot}
    \dot{\Lambda}_5 = -\int_{\Omega}\dot{V} v |u|^p ~dx - p \int_{\Omega} V|u|^{p-2}u (\nabla u \cdot v)~dx +\int_{\partial \Omega}V |u|^p (v \cdot \eta ) ~dx.
\end{equation}

\noindent It can be seen that $\dot{\Lambda}_1^\varepsilon$ and $\dot{\Lambda}_2^\varepsilon$ converges to $\dot{\Lambda}_1$ and $\dot{\Lambda}_2,$ respectively as $\varepsilon \to 0.$ Therefore,   
\begin{equation}
\label{First formula for lambda dot 0}
\begin{aligned}
 \dot{\lambda}_1(0)&:=\displaystyle{\lim_{\varepsilon \to 0}}\;\dot{\Lambda}_1^\varepsilon+\dot{\Lambda}_2^\varepsilon+\dot{\Lambda}_3+\dot{\Lambda}_4 + \dot{\Lambda}_5 =\dot{\Lambda}_1+\dot{\Lambda}_2+\dot{\Lambda}_3+\dot{\Lambda}_4 + \dot{\Lambda}_5\\
&=\int_{\partial \Omega} \{\left[|\nabla u|^p - \lambda_1 (0) |u|^p + \beta |u|^p (n-1)H + V |u|^p \right] (v \cdot \eta) \\
&\quad \quad -p |\nabla u|^{p-2}(\nabla u \cdot v)(\nabla u \cdot \eta) + \beta |u|^p  \text{div}_{\partial\Omega} v^\tau  \} ~ds. 
\end{aligned} 
\end{equation}
\noindent Substituting from the boundary condition of (\ref{Euler Lagrange equation in Omega_t}) in (\ref{First formula for lambda dot 0}),
\begin{equation}
\label{Third formula for lambda dot 0}
\begin{aligned}
 \dot{\lambda}_1(0)&=\int_{\partial \Omega} \{\left[|\nabla u|^p - \lambda_1 (0) |u|^p + \beta |u|^p (n-1)H +V |u|^p \right] (v \cdot \eta) \\
&\quad \quad +p \beta |u|^{p-2} u (\nabla u \cdot v) + \beta |u|^p  \text{div}_{\partial\Omega} v^\tau  \} ~ds. 
\end{aligned} 
\end{equation}

\noindent Note that 
$$ \begin{aligned}
|u|^{p-2} (\nabla u \cdot v )&= \left( \left( |u|^{p-2} \nabla u  \right)^\tau +  \left( |u|^{p-2} \nabla u \cdot \eta \right) \eta \right) \cdot \left( v^\tau + (v \cdot \eta) \eta \right) \\
&= \left( |u|^{p-2} \nabla u  \right)^\tau \cdot v^\tau + \left( |u|^{p-2} \nabla u \cdot \eta \right) (v \cdot \eta).
\end{aligned} $$
Using the above equation in \eqref{Third formula for lambda dot 0}, we get  
\begin{equation}
\label{Fourth formula for lambda dot 0}
\begin{aligned}
  \dot{\lambda}_1(0)&=\int_{\partial \Omega} \{\left[|\nabla u|^p - \lambda_1 (0) |u|^p + \beta |u|^p (n-1)H + p \beta u \left( |u|^{p-2} \nabla u \cdot \eta \right) +V |u|^p\right] (v \cdot \eta) \\
&\quad \quad +p \beta u \left( |u|^{p-2} \nabla u  \right)^\tau \cdot v^\tau + \beta |u|^p  \text{div}_{\partial\Omega} v^\tau  \} ~ds. 
\end{aligned} 
\end{equation}
By the Gauss theorem on surfaces (\ref{Gauss theorem on surfaces}), 
	 \begin{equation*}
	 \begin{aligned}
	 \int_{\partial \Omega}  \beta |u|^p  \text{div}_{\partial\Omega} v^\tau ~ds &=- \beta\int_{\partial \Omega}v^\tau \cdot \nabla ^\tau |u|^p ~ds + \beta(n-1) \int_{\partial \Omega}|u|^p (v^\tau\cdot\eta) H ~ds \\
	 &=- \beta\int_{\partial \Omega}v^\tau \cdot \nabla ^\tau |u|^p ~ds.
	 \end{aligned}
	 \end{equation*}
The term $\nabla ^\tau |u|^p$ can be reformulated using (\ref{grad tau f on boundary}) as
	\begin{equation*}
	 \nabla ^\tau |u|^p =g^{ij} (\nabla |u|^p \cdot x_{\xi_i}) x_{\xi_j} =g^{ij}(p|u|^{p-2}u\nabla u \cdot x_{\xi_i}) x_{\xi_j} =pu(|u|^{p-2}\nabla u)^\tau.
	 \end{equation*}
	 Thus \begin{equation}
	 \label{For cancelling the extra terms in lambda dot 0}
	    \int_{\partial \Omega}  \beta |u|^p  \text{div}_{\partial\Omega} v^\tau ~ds = - \int_{\partial \Omega}p \beta u v^\tau \cdot  (|u|^{p-2}\nabla u)^\tau ~ds.
	 \end{equation}
The equation  (\ref{Fourth formula for lambda dot 0}) can be simplified using (\ref{For cancelling the extra terms in lambda dot 0}) to arrive at the following expression.
	 \begin{equation*}
	 \label{Final formula for lambda dot 0 -1}
	      \dot{\lambda}_1(0)=\int_{\partial \Omega} \left[|\nabla u|^p - \lambda_1 (0) |u|^p + \beta |u|^p (n-1)H + p \beta u \left( |u|^{p-2} \nabla u \cdot \eta \right) +V |u|^p \right] (v \cdot \eta)   ~ds.
	 \end{equation*}
The above expression can be rephrased as follows using the boundary condition of (\ref{Euler Lagrange equation in Omega_t}).
\begin{equation*}
\label{Final formula for lambda dot 0 -2}
  \dot{\lambda}_1(0)=\int_{\partial \Omega}   \left[|\nabla u|^p - \lambda_1 (0) |u|^p + \beta |u|^p (n-1)H - p \beta^2 \frac{|u|^{2p-2}}{|\nabla u|^{p-2}} +V |u|^p \right] (v \cdot \eta) ~ds. 
\end{equation*}
Hence 
\begin{equation}
        \dot{\lambda}_1(0)  =\int_{\partial \Omega} \left[|\nabla u|^p - \lambda_1 (0) |u|^p + \beta |u|^p (n-1)H + p \beta u \left( |u|^{p-2} \nabla u \cdot \eta \right) +V |u|^p \right] (v.\eta)   ds \\
\end{equation}

The above discussion is summarized in Theorem \ref{Shape derivative formula-Schrodinger}.

\begin{remark} The following domain monotonicity results hold for $\lambda_1(V).$ 
\begin{itemize}
    \item [(i)] Let $\Omega  = B(0,R)$ in $\mathbb{R}^n.$ Assume $V$ is radial and $V(x)\leq \lambda_1 \; a.e. $ on $\partial \Omega.$
        \begin{enumerate}
             \item [(a)] If $\beta \geq \left(\frac{n-1}{R(p-1)}\right)^{p-1},$ then $\dot{\lambda}_1(0)<0 \; (>0)$ for a smooth perturbation $v$ of the domain  with $\int_{\partial\Omega}v \cdot \eta ~ds >0 \; (<0)$ on $\partial \Omega$.

             \item [(b)] Let $R \geq 1,\beta \geq 1$ and $p \geq n.$ Then $\dot{\lambda}_1(0)\leq 0 \; (\geq 0)$ for a smooth $v$  perturbation of the domain with $\int_{\partial\Omega}v \cdot \eta ~ds >0 \; (<0)$ on $\partial \Omega$. If we assume a strict inequality in any one of the inequalities for $R, \beta, $ or $p,$ then we get a strict inequality for $\dot{\lambda}_1(0).$
         \end{enumerate} 

         \item [(ii)] Let $\Omega \subseteq \mathbb{R}^n$ be a bounded domain in the H$\ddot{\text{o}}$lder class $C^{2,\alpha},$ $\Omega_t$ be a family of perturbations with $ \Omega_t :=\{y=x+tv(x)+o(t):x\in \Omega, |t| \text{ sufficiently small}\}$ where $v=(v_1(x),v_2(x),..., v_n(x)) $ is a $C^{2,\alpha}$ vector field such that $v(x)\cdot \eta(x)>0$ for all $x\in \partial\Omega$. If $V(x)\leq \lambda_1 \; a.e. $ on $\partial \Omega,$ then there exists a constant $\beta^*(\Omega, v)>0$ such that  for $\beta>\beta^*$, $\dot{\lambda}_1(0)<0$.
    \end{itemize}
The proofs are similar to those of Theorem 1.3 and Theorem 1.4 of \cite{mallick2022shape}.
    
\end{remark}

\bibliography{References}
\bibliographystyle{amsplain}

\textsc{Ardra A}\\
Department of Mathematics,\\
Indian Institute of Technology Palakkad,\\
Kerala-678623, India \\
ardra.math@gmail.com
\end{document}